\documentclass[12pt,reqno]{amsart}
\usepackage{amsmath,amssymb,geometry,color,tikz-cd}

\usepackage[backref,pagebackref,pdftex,hyperindex]{hyperref}
\usepackage[alphabetic]{amsrefs}
\usepackage{amssymb}% see geometry.pdf on how to lay out the page. There's lots.
\usepackage{bbm}
\usepackage{enumitem}
\usepackage{upgreek}

\newtheorem{proposition}{Proposition}[section]
\newtheorem{theorem}[proposition]{Theorem}
\newtheorem{lemma}[proposition]{Lemma}
\newtheorem{corollary}[proposition]{Corollary}

\theoremstyle{definition}
\newtheorem{remark}[proposition]{Remark}
\newtheorem{definition}[proposition]{Definition}
\newtheorem{example}[proposition]{Example}
\newtheorem{question}[proposition]{Question}

\title{Two results regarding the variation of K-moduli}
\author{Fei Si, Zheng Zhang, Chuyu Zhou}

\address{The School of Mathematics and Statistics, Xi’an Jiaotong University, 28 West Xianning Road, Xi’an, Shaanxi 710049, China}
\email{sifei@xjtu.edu.cn}

\address{Institute of Mathematical Sciences, ShanghaiTech University, 393 Middle Huaxia Road, Shanghai 201210, China}
\email{zhangzheng@shanghaitech.edu.cn}

\address{School of Mathematical Sciences, Xiamen University, Siming South Road 422, Xiamen, Fujian 361005, China}
\email{chuyuzhou1@gmail.com}
\date{} % delete this line to display the current date
\thanks{2010 
	    \emph{Mathematics Subject Classification}: 14J45.
	    \newline
	    \indent 
		\emph{Keywords}: Log Fano pair, K-stability, K-semistable domain, K-moduli, Variation of GIT.
        \newline
		\indent
		\emph{Competing interests}: The authors declare none.
		}

\newcommand{\Fut}{{\rm{Fut}}}

\newcommand{\ord}{{\rm {ord}}}
\newcommand{\tc}{{\rm {tc}}}
\newcommand{\vol}{{\rm {vol}}}

\newcommand{\Proj}{{\rm{Proj}}}

\newcommand{\PGL}{{\rm {PGL}}}

\newcommand{\Aut}{{\rm {Aut}}}

\newcommand{\CM}{{\rm {CM}}}

\newcommand{\Kss}{{\rm {Kss}}}

\newcommand{\GIT}{{\rm {GIT}}}
\newcommand{\Diff}{{\rm {Diff}}}

\newcommand{\bA}{\mathbb{A}}

\newcommand{\bC}{\mathbb{C}}

\newcommand{\bN}{\mathbb{N}}

\newcommand{\bP}{\mathbb{P}}
\newcommand{\bQ}{\mathbb{Q}}

\newcommand{\bZ}{\mathbb{Z}}

\newcommand{\mB}{\mathcal{B}}

\newcommand{\mD}{\mathcal{D}}
\newcommand{\mE}{\mathcal{E}}

\newcommand{\mG}{\mathcal{G}}

\newcommand{\mL}{\mathcal{L}}
\newcommand{\mM}{\mathcal{M}}

\newcommand{\mO}{\mathcal{O}}

\newcommand{\mS}{\mathcal{S}}

\newcommand{\mU}{\mathcal{U}}

\newcommand{\mX}{\mathcal{X}}
\newcommand{\mY}{\mathcal{Y}}

\begin{document}

\begin{abstract}
In this note, we prove two results regarding the variation of K-moduli. The first one reveals the relationship between the chamber decomposition for K-semistable domains and the variation of GIT. The second one presents the relationship between the K-moduli generically parametrizing K-semistable smooth Fano complete intersections of the form $S_1\cap...\cap S_k$ and the K-moduli generically parametrizing K-semistable log Fano manifolds of the form $(\mathbb{P}^n, \sum_{j=1}^kx_jS_j)$, where $x_j\in (0,1)\cap \mathbb{Q}$ and $S_j\subset \mathbb{P}^n$ is a hypersurface of degree $d_j$ for each $1\leq j\leq k$.
\end{abstract}

\maketitle

\setcounter{tocdepth}{1}

%\tableofcontents

\section{Introduction}

We work over the complex number field $\bC$ throughout the article.

Based on a rather complete algebraic K-stability theory developed in the past few years (e.g. \cite{Jiang20, BLX22, Xu20, ABHLX20, BX19, XZ20b, LXZ22}),  people are now able to  construct a projective scheme as a good moduli space to parametrize K-polystable Fano varieties. In particular, a wall crossing theory for K-moduli spaces is also established in \cite{ADL19, Zhou23}, which turns out to be useful in the study of birational geometry of different moduli spaces of algebraic varieties (e.g. \cite{ADL20, ADL21}). 
In the wall crossing picture, it is widely recognized that the K-moduli space before the first wall is related to the GIT moduli space (e.g. \cite{ADL19, GMGS21, Zhou21a}). For example, fixing a K-polystable Fano variety $X$ and a sufficiently divisible positive integer $l$, there is a natural $\Aut(X)$-action on the linear system $|-lK_X|$ (note that $\Aut(X)$ is reductive due to \cite{ABHLX20}). We conclude that $D\in |-lK_X|$ is GIT-(semi/poly)stable if and only if $(X, \frac{\epsilon}{l} D)$ is K-(semi/poly)stable for $0<\epsilon\ll 1$ (e.g. \cite[Theorem 1.1]{Zhou21a}, see also \cite{ADL19, GMGS21}). Recently, a wall crossing theory for K-moduli with multiple boundaries is established in \cite{Zhou23b}, therefore, it is also natural to determine the similar relationship (as in the above example) in the setting of variation of GIT.

To state the first result, we fix a K-polystable Fano variety $X$ and $k$ ample line bundles $L_1,...,L_k$ such that $L_j\sim_\bQ -l_jK_X$ for some rational $l_j>0$. Denote by $\bP_j:=|L_j|,\ j=1,...,k$. Then there is a natural $\Aut(X)$-action on $\bP_1\times...\times \bP_k$ and one could consider the variation of GIT under the linearization $\mO(\alpha_1,...,\alpha_k)$ as $(\alpha_1,...,\alpha_k)$ varies in $\bQ_+^k$ (e.g. \cite{Tha96}). On the other hand, by \cite{Zhou23b}, there is a finite rational polytope chamber decomposition of $[0,1]^k$
such that for any 
$$(D_1,...,D_k)\in \bP_1\times...\times \bP_k,$$ 
the K-semistability of $(X, \sum_{j=1}^kx_jD_j)$ does not change as $(x_1,...,x_k)$ varies in the interior domain of each chamber. We have the following result which reveals the relationship between the variation of GIT and the chamber decomposition of $[0,1]^k$.

\begin{theorem}\label{thm: main1}
Fix a rational vector $(c_1,...,c_k)\in (0,1)^k$. Suppose $(X, r\sum_{j=1}^kc_jB_j)$ is a K-semistable log Fano pair for some $(B_1,...,B_k)\in \bP_1\times...\times \bP_k$ and some rational $r>0$. Then $(D_1,...,D_k)\in \bP_1\times...\times \bP_k$ 
is GIT-(semi/poly)stable under $\Aut(X)$-action with respect to the linearization $\mO(c_1,...,c_k)$ if and only if $(X, \epsilon\sum_{j=1}^kc_jD_j)$ is K-(semi/poly)stable for $0<\epsilon\ll 1$. In other words, the chamber decomposition of $[0,1]^k$ near the original point $(0,0,...,0)\in [0,1]^k$ exactly gives the chamber decomposition for the variation of GIT.
\end{theorem}

The above result is a natural generalization of \cite[Theorem 1.1]{Zhou21a} and concerns the case where the coefficients are close to $0$. We now turn to the second result, which discusses the behavior near another extremal point of the K-semistable domain. Fix a positive integer $n$ and $k$ positive integers $\vec{d}:=(d_1,...,d_k)$. For log smooth hypersurfaces $S_j$'s of degrees $d_j$'s, if $\sum_{j=1}^k d_j< n+1$, then the complete intersection $\cap_{j=1}^kS_j$ gives a Fano manifold of dimension $n-k$. We denote $\mM^K_{n, \vec{d}}$ (resp. $M^K_{n, \vec{d}}$) to be the K-moduli stack (resp. K-moduli space) generically parametrizing such Fano manifolds which are K-semistable.\footnote{It is conjectured that the complete intersection $\cap_{j=1}^kS_j$ is K-semistable if $S_j, j=1,...,k,$ are general with $\sum_{j=1}^kd_j<n+1$. However, this conjecture is still open as far as we know. Up to this point, the stack $\mM^K_{n, \vec{d}}$ might be empty.} Let $\mM^K_{n, \vec{d}, \vec{x}}$ (resp. $M^K_{n, \vec{d}, \vec{x}}$) be the K-moduli stack (resp. K-moduli space) generically parametrizing K-semistable log Fano manifolds of the form  $(\bP^n, \sum_{j=1}^kx_jS_j)$, where $\vec{x}:=(x_1,...,x_k)$ and $x_j>0$ for each $j$. Then it is natural to ask what is the relationship
between $\mM^K_{n, \vec{d}}$ and $\mM^K_{n, \vec{d}, \vec{x}}$. First, recall the following result.

\begin{theorem}{\rm{(\cite[Theorem 8.1]{Zhou23a})}}\label{thm: domain}
Consider the log smooth pair $(\bP^n, S_1+S_2+...+S_k)$, where $S_j, j=1,...,k,$ are mutually distinct smooth hypersurfaces in $\bP^n$ of degrees $d_j$ with $n\geq 2$. Suppose all the Fano complete intersections are K-semistable\footnote{This means that for any non-repeating subset $\{i_1,...,i_l\}\subset \{1,2,...,k\}$ such that $\sum_{p=1}^ld_{i_p}<n+1$, the complete intersection $\cap_{p=1}^l S_{i_p}$ is K-semistable.}. Then the K-semistable domain\footnote{See Definition \ref{def: kss domain}.} $\Kss(\bP^n, S_1+...+S_k)$ is a rational polytope generated by the following equations
\begin{equation*}
\begin{cases}
0\leq x_i\leq 1, \ \ \  1\leq i\leq k\\

\beta_{\bP^n, \sum_{j=1}^kx_jS_j}(S_i)\geq 0, \ \ \  1\leq i\leq k\\

\sum_{j=1}^k x_jd_j\leq n+1.
\end{cases}
\end{equation*}
\end{theorem}

By the proof of \cite[Theorem 8.1]{Zhou23a}, if $\sum_{j=1}^kd_j< n+1$, the following equations 
$$\beta_{\bP^n, \sum_{j=1}^kx_jS_j}(S_i)= 0, \ \ \ 1\leq i\leq k$$ 
admit a unique solution, denoted by $\vec{a}:=(a_1,...,a_k)$, where
$$a_j=\frac{\sum_{i=1}^kd_i+(n-k+1)d_j-(n+1)}{(n-k+1)d_j}, $$
and $\vec{a}$ is automatically an extremal point of $\Kss(\bP^n, \sum_{j=1}^kS_{d_j})$. Given $\sum_{j=1}^k d_j<n+1$,  it is clear that $0\leq a_j<1$, and $a_j=0$ if and only if $d_1=d_2=...=d_k=1$. As an extremal point, $\vec{a}=(a_1,...,a_k)$ is definitely distinct from other points in $\Kss(\bP^n, \sum_{j=1}^kS_{d_j})$. We have the following second main result.

\begin{theorem}\label{thm: main2}
Let $\vec{a}=(a_1,...,a_k)$ be the unique solution as above. For a log smooth pair $(\bP^n, \sum_{j=1}^kS_j)$, where $S_j$'s are distinct hypersurfaces of degrees $d_j$'s 
%{\rm{(}}not necessarily in general position{\rm{)}}   
with $\sum_{j=1}^kd_j<n+1$, we have the following conclusions:
\begin{enumerate}
\item the log Fano pair $(\bP^n, \sum_{j=1}^ka_jS_j)$ is K-semistable if and only if the complete intersection $\cap_{j=1}^{k}S_j$ is  K-semistable;
\item there exists a morphism $\mM^{K}_{n, \vec{d}, \vec{a}}\to \mM^K_{n, \vec{d}}$ 
which descends to a surjective morphism $M^{K}_{n, \vec{d}, \vec{a}}\to M^K_{n, \vec{d}}.$
\end{enumerate}
\end{theorem}

\begin{remark}
The above theorem can be generalized to the case where $\bP^n$ is replaced by a K-polystable Fano manifold and the boundaries are smooth divisors proportional to the anti-canonical divisor. The idea and the approach are similar. 
\end{remark}

After recalling the necessary preliminaries in Section \ref{section:preliminary}, we prove Theorem \ref{thm: main1} in Section \ref{sec: main1} and Theorem \ref{thm: main2} in Section \ref{sec: main2}. Finally, in Section \ref{sec: cubic-hyperplane}, we present a cubic-hyperplane example in $\bP^4$ to illustrate our main results. A closely related K-moduli space will be studied in detail in a forthcoming paper, in which the results of this note will play an important role.

\noindent
\subsection*{Acknowledgement}
F.S is partially supported by the Fundamental Research Funds for the Central Universities and Shaanxi NSF (No. 2025JC-QYCX-002). Z.Z is supported in part by the NSFC grant (No. 12201406). C.Z is supported by the NSFC grant (No. 12501058) and a grant from Xiamen University (No. X2450214).

\section{Preliminaries} \label{section:preliminary}

We say that $(X,\Delta)$ is a \emph{log pair} if $X$ is a normal projective variety and $\Delta$ is an effective $\bQ$-divisor on $X$ such that $K_X+\Delta$ is $\bQ$-Cartier.  The log pair $(X,\Delta)$ is called \emph{log Fano} if it admits klt singularities and $-(K_X+\Delta)$ is ample; if $\Delta=0$, we just say $X$ is a \emph{Fano variety}. 
For various types of singularities in birational geometry, e.g.  klt, lc, and plt singularities, we refer to \cite{KM98,Kollar13}.

\subsection{K-stability}

Let $(X,\Delta)$ be a log pair. Suppose $f\colon Y\to X$ is a proper birational morphism between normal varieties and $E$ is a prime divisor on $Y$, we say that $E$ is a prime divisor over $X$ and define the following invariant
$$A_{X,\Delta}(E):=1+\ord_E(K_Y-f^*(K_X+\Delta)). $$
It is called the \emph{log discrepancy} of $E$ associated to the log pair $(X,\Delta)$.
If $(X,\Delta)$ is a log Fano pair, we define the following invariant
$$S_{X,\Delta}(E):=\frac{1}{\vol(-K_X-\Delta)}\int_0^\infty \vol(-f^*(K_X+\Delta)-tE){\rm{d}}t .$$
Put $\beta_{X,\Delta}(E):=A_{X,\Delta}(E)-S_{X,\Delta}(E)$. By the works \cite{Fuj19, Li17}, one can define K-stability of a log Fano pair by the following criterion.
\begin{definition}\label{def: kss}
Let $(X,\Delta)$ be a log Fano pair. 
We say that $(X,\Delta)$ is \emph{K-semistable} if $\beta_{X,\Delta}(E)\geq 0$ for any prime divisor $E$ over $X$.
\end{definition}

The following lemma on K-stability of projective cones is well known, see e.g. \cite[Prop 2.11]{LZ22}.
\begin{lemma}\label{lem:cone stability}
Let $(V,\Delta)$ be an n-dimensional log Fano pair, and L an ample line bundle on V such that $L\sim_\bQ -\frac{1}{r}(K_V+\Delta)$ for some $0<r\leq n+1$. Suppose Y is the projective cone over V associated to L with infinite divisor $V_\infty$, then $(V,\Delta)$ is K-semistable if and only if $(Y,\Delta_Y+(1-\frac{r}{n+1})V_\infty)$ is K-semistable, where $\Delta_Y$ is the divisor on Y naturally extended by $\Delta$.
\end{lemma}

\subsection{Test configuration}

\begin{definition}\label{def: tc}
Let $(X,\Delta)$ be a log pair and $L$ an ample $\bQ$-line bundle on $X$. A test configuration $\pi: (\mX,\Delta_\tc;\mL)\to \bA^1$ is a degenerating family over $\bA^1$ consisting of the following data:
\begin{enumerate}
\item $\pi: \mX\to \bA^1$ is a projective flat morphism from a normal variety $\mX$, $\Delta_\tc$ is an effective $\bQ$-divisor on $\mX$ which does not contain any component in the central fiber $\mX_0$, and $\mL$ is a relatively ample $\bQ$-line bundle on $\mX$,
\item the family $\pi$ admits a $\bC^*$-action which lifts the natural $\bC^*$-action on $\bA^1$ such that $(\mX,\Delta_\tc; \mL)\times_{\bA^1}\bC^*$ is $\bC^*$-equivariantly isomorphic to $(X, \Delta; L)\times_{\bA^1}\bC^*$.
\end{enumerate}
\end{definition}

Suppose $(X,\Delta)$ is a log Fano pair and $L=-K_X-\Delta$. Let  $(\mX,\Delta_\tc; \mL)$ be a test configuration such that $\mL=-K_{\mX/\bA^1}-\Delta_\tc$. We call it a special test configuration if  $(\mX, \mX_0+\Delta_{\tc})$ admits plt singularities (or equivalently, the central fiber $(\mX_0, \Delta_{\tc,0})$ is a log Fano pair).

\begin{remark}\label{rem: kss degeneration}
Let $(X, \Delta)$ be a K-semistable log Fano pair. Suppose $S$ is a prime divisor on $X$ such that $\beta_{X, \Delta}(S)=0$, then $S$ induces a special test configuration with vanishing generalized Futaki invariant (e.g. \cite{LX14, Fuj19}). By \cite{LWX21}, we know that the central fiber of the test configuration is also K-semistable. Assume further that $S$ is Cartier and $S\sim_\bQ -\lambda(K_X+\Delta)$ for some rational $0<\lambda<1$, then it is not hard to see that 
the central fiber of the test configuration is isomorphic to the projective cone over $S$ with respect to the polarization $S|_S$. Actually, first note that 
$$X=\Proj \bigoplus_{m\in \bN} H^0(X, mS). $$
By \cite[Section 2.3.1]{BX19},  the special test configuration induced by $S$ could be formulated as
$$\mX:=\Proj \bigoplus_{m\in \bN}\bigoplus_{i\in \bZ} H^0(X, mS-iS)t^{-i}, $$
and the central fiber is
$$\mX_0=\Proj \bigoplus_{m\in \bN}\bigoplus_{i\in \bZ} \left(H^0(X, mS-iS)/H^0(X, mS-(i+1)S)\right) .$$
Note that for $i\leq m$ we have
$$H^0(X, mS-iS)/H^0(X, mS-(i+1)S)\cong  H^0(S, (m-i)S|_S).$$
Thus, $\mX_0\cong \Proj \bigoplus_{m\in \bN}\bigoplus_{i\in \bN}H^0(S, mS|_S)s^i$, which is the projective cone over $S$ with respect to $S|_S$.
\end{remark}

\section{Variation of K-moduli and variation of GIT-moduli}\label{sec: main1}

In this section, we aim to prove Theorem \ref{thm: main1}.

\subsection{Chow-Mumford line bundle}

Let $\pi: (\mX,\mD;\mL)\to T$ be a flat family of projective normal varieties of dimension $n$ over a proper normal base $T$, where $\mD$ is an effective $\bQ$-divisor on $\mX$ whose components are all flat over $T$, and $\mL$ is a relatively ample $\bQ$-line bundle on $\mX$. By the work of Mumford-Knudsen(\cite{KM76}), there exist $\bQ$-line bundles $\lambda_i$ ($i=0,1,...,n+1$) and $\tilde{\lambda}_i$ ($i=0,1,...n$) on $T$ such that we have the following expansions for all sufficiently large $k\in \bN$:
$$\det \pi_*(\mL^k)= \lambda_{n+1}^{\binom{k}{n+1}}\otimes\lambda_n^{\binom{k}{n}}\otimes...\otimes\lambda_1^{\binom{k}{1}}\otimes\lambda_0,$$
$$\det \pi_*(\mL|_\mD^k)= \tilde{\lambda}_{n}^{\binom{k}{n}}\otimes\tilde{\lambda}_{n-1}^{\binom{k}{n-1}}\otimes...\otimes\tilde{\lambda}_0.$$
By Riemann-Roch formula, cf \cite[Appendix]{CP21}, we have
$$c_1(\pi_*\mL^k)=\frac{\pi_*(\mL^{n+1})}{(n+1)!}k^{n+1}+\frac{\pi_*(-K_{\mX/T}\mL^n)}{2n!}k^n+... ,$$
$$c_1({\pi}_*{\mL}|_\mD^k)=\frac{{\pi}_*(\mL^n\mD)}{n!}k^n+... .$$
Then it is not hard to see
$$\lambda_{n+1}=\pi_*(\mL^{n+1}),\  \lambda_n=\frac{n}{2}\pi_*(\mL^{n+1})+\frac{1}{2}\pi_*(-K_{\mX/T}\mL^n),\  
\tilde{\lambda}_n=\pi_*(\mL^n\mD). $$
By the flatness of $\pi$ and $\pi_\mD$, we write
$$h^0(\mX_t, k\mL_t)=a_0k^n+a_1k^{n-1}+o(k^{n-1})\quad \text{and} \quad  h^0(\mD_t, k{\mL_t}|_{\mD_t})=\tilde{a}_0k^{n-1}+o(k^{n-1}),$$
which do not depend on the choice of $t\in T$. Then we have
$$a_0=\frac{\mL_t^n}{n!},\  a_1=\frac{-K_{\mX_t}{\mL_t}^{n-1}}{2(n-1)!},\  \tilde{a}_0=\frac{\mL_t^{n-1}\mD_t}{(n-1)!}.$$

\begin{definition}\label{def: CM}
We define the CM-line bundles for the family $\pi: (\mX,\mD;\mL)\to T$ as follows:
$$\lambda_{\CM}(\mX,\mL;\pi):=\lambda_{n+1}^{\frac{2a_1}{a_0}+n(n+1)}\otimes\lambda_n^{-2(n+1)}, $$
$$\lambda_{\CM}(\mX,\mD,\mL;\pi):= \lambda_{n+1}^{\frac{2a_1-\tilde{a}_0}{a_0}+n(n+1)}\otimes\lambda_n^{-2(n+1)}\otimes\tilde{\lambda}_n^{n+1}.$$
For a rational number $0\leq \beta\leq 1$, we define the following
$$\lambda_{\CM, \beta}(\mX,\mD,\mL;\pi):= \lambda_{n+1}^{\frac{2a_1-\beta\tilde{a}_0}{a_0}+n(n+1)}\otimes\lambda_n^{-2(n+1)}\otimes\tilde{\lambda}_n^{\beta(n+1)}.$$
\end{definition}

\begin{remark}\label{rem: futaki}
If $\pi: (\mX, \mD; \mL)\to T$ is a test configuration of the general fiber $(\mX_t, \mD_t; \mL_t)$, then it is well known that the generalized Futaki invariant $\Fut(\mX, \beta\mD; \mL)$ is equal to $w(\lambda_{\CM, \beta}(\mX,\mD,\mL;\pi))$ up to a positive factor, where $w(\lambda_{\CM, \beta}(\mX,\mD,\mL;\pi))$ is the total weight of the $\bC^*$-action on $\lambda_{\CM, \beta}(\mX,\mD,\mL;\pi)$ (e.g. \cite[Theorem 2.6]{GMGS21}).
\end{remark}

\begin{lemma}\label{lem: proportional cm}
Let $\pi: (\mX, \mD; \mL)\to T$ be a family that satisfies:
\begin{enumerate}
\item $\mX\to T$ is a flat family of projective normal varieties of dimension $n$ over a proper normal base $T$ such that $-K_{\mX/T}$ is $\bQ$-Cartier and relative ample over $T$; 
\item $\mD$ is an effective $\bQ$-divisor on $\mX$ such that every component is flat over $T$ and $\mD\sim_{\bQ, T}-\mu K_{\mX/T}$ for some rational $\mu>0$;
\item $\mL=- K_{\mX/T}$.
\end{enumerate}
\end{lemma}
Then we have the following formula
$$\lambda_{\CM,\beta}(\mX, \mD,\mL;\pi)=-(1+\beta \mu n)\pi_*\mL^{n+1}+\beta (n+1) \pi_*(\mL^n\mD).$$

\begin{proof}
Applying Definition \ref{def: CM}, we have
\begin{align*}
&\lambda_{\CM, \beta}(\mX,\mD,\mL;\pi) \\
=\ &\lambda_{n+1}^{\frac{2a_1-\beta\tilde{a}_0}{a_0}+n(n+1)}\otimes\lambda_n^{-2(n+1)}\otimes\tilde{\lambda}_n^{\beta(n+1)}\\
=\ & \left(\frac{}{}n(1-\beta \mu)+n(n+1)\right)\pi_*\mL^{n+1}\\
\ &-2(n+1)\left(\frac{n}{2}\pi_*(\mL^{n+1})+\frac{1}{2}\pi_*(-K_{\mX/T}\mL^n) \right)+\beta (n+1) \pi_*(\mL^n\mD)\\
=\ &-(1+\beta \mu n)\pi_*\mL^{n+1}+\beta (n+1) \pi_*(\mL^n\mD).
\end{align*}
The proof is complete.
\end{proof}

\subsection{Chamber decomposition for K-moduli and GIT-moduli}

In this subsection, we fix a K-polystable Fano variety $X$ of dimension $n$ with $(-K_X)^n=v$ and $k$ ample line bundles $L_1,...,L_k$ such that $L_j\sim_\bQ -l_jK_X$ for some rational $l_j>0$. Put 
$$\bP_j:=|L_j|,\ j=1,...,k.$$
Let $\mD_j\subset X\times \bP_j$ be the universal divisor with respect to the linear system $|L_j|$. Denote 
$$T:=\bP_1\times ... \times \bP_k\quad \text{and} \quad\mB_j:=\bP_1\times...\times\bP_{j-1}\times \mD_j\times \bP_{j+1}\times...\times\bP_k.$$
For a given rational point $\vec{c}:=(c_1,...,c_k)\in (0,1)^k$, our aim is to compute the Chow-Mumford line bundle 
$$\lambda_{\CM, \beta}(\vec{c}):=\lambda_{\CM,\beta}(X\times T, \sum_{j=1}^k c_j\mB_j, -K_{X\times T/T}; \pi_{\vec{c}})$$ 
for the family 
$\pi_{\vec{c}}: (X\times T, \sum_{j=1}^k c_j\mB_j)\to T$
for rational $0\leq \beta\leq 1$.

\begin{lemma}\label{lem: compute cm}
Notation as above, for a rational point $\vec{c}:=(c_1,...,c_k)\in (0,1)^k$, we have the following formula:
$$\lambda_{\CM, \beta}(\vec{c})= \mO(\beta (n+1)vc_1,...,\beta(n+1)vc_k).$$
\end{lemma}

\begin{proof}
First, note that $\sum_{j=1}^k c_j\mB_j\sim_{\bQ, T} -(\sum_{j=1}^k c_jl_j)K_{X\times T/T}$.
By Lemma \ref{lem: proportional cm}, we have
\begin{align*}
\lambda_{\CM, \beta}(\vec{c})&=-(1+\beta n\sum_{j=1}^kc_jl_j){\pi_{\vec{c}}}_*(-K_{X\times T/T})^{n+1}+ (1+n)\beta{\pi_{\vec{c}}}_*\left((-K_{X\times T/T})^n.\sum_{j=1}^kc_j\mB_j\right)\\
&=\sum_{j=1}^k (n+1)\beta c_j {\pi_{\vec{c}}}_*((-K_{X\times T/T})^n\mB_j). 
\end{align*}
It suffices to compute ${\pi_{\vec{c}}}_*((-K_{X\times T/T})^n\mB_j)$. Recall that 
$$\mB_j:=\bP_1\times...\times\bP_{j-1}\times \mD_j\times \bP_{j+1}\times...\times\bP_k,$$ 
and 
${\pi_j}_*((-K_{X\times \bP_j/\bP_j})^n\mD_j)=\mO_{\bP_j}(v)$
for $\pi_{j}: X\times \bP_j\to \bP_j$. Thus, we have
$${\pi_{\vec{c}}}_*((-K_{X\times T/T})^n\mB_j)=\mO(0_1,0_2...0_{j-1}, v,0_{j+1},...,0_k),$$
where $0_i$ means that the $i$-th position is $0$. The proof is finished.
\end{proof}

It is not difficult to see that $\lambda_{\CM, \beta}(\vec{c})$ is proportional to $\mO(c_1,...,c_k)$.

\begin{proposition}\label{prop: GIT=K}
Fix a rational vector $(c_1,...,c_k)\in (0,1)^k$. Suppose $(X, r\sum_{j=1}^kc_jB_j)$ is a K-semistable log Fano pair for some $(B_1,...,B_k)\in \bP_1\times...\times \bP_k$ and some rational $r>0$. Then $(D_1,...,D_k)\in \bP_1\times...\times \bP_k$ 
is GIT-(semi/poly)stable under $\Aut(X)$-action with respect to the linearization $\mO(c_1,...,c_k)$ if and only if $(X, \epsilon\sum_{j=1}^kc_jD_j)$ is K-(semi/poly)stable for $0<\epsilon\ll 1$. 
\end{proposition}

\begin{proof}
The idea of the proof is essentially the same as that of \cite[Theorem 1.1]{Zhou21a}. 

Suppose $(X, \epsilon\sum_{j=1}^kc_jD_j)$ is K-semistable, to show that $(D_1,...,D_k)$ is GIT-semistable under the $\Aut(X)$-action with respect to the linearization $\mO(c_1,...,c_k)$, by Remark \ref{rem: futaki}, it suffices to show that $\lambda_{\CM, \epsilon}(\vec{c})$ is proportional to $\mO(c_1, ..., c_k)$. This is exactly Lemma \ref{lem: compute cm}. Now we suppose $(X, \epsilon\sum_{j=1}^kc_jD_j)$ is K-polystable, then $(D_1,...,D_k)$ is GIT-semistable as we have shown. Let $\rho: \bC^*\to \Aut(X)$ be any one parameter subgroup such that the limit point  $\lim_{t\to 0}\rho(t)\cdot (D_1,...,D_k)$ is GIT-polystable, then the GIT-weight corresponding to the $\bC^*$-action is zero, which implies that the Futaki invariant of the test configuration induced by the $\bC^*$-action vanishes. By \cite{LWX21}, we see that 
$$\lim_{t\to 0}\rho(t)\cdot (D_1,...,D_k)\cong (D_1,...,D_k).$$ 
Thus $(D_1,...,D_k)$ is GIT-polystable.

Conversely, suppose $(D_1,...,D_k)$ is GIT-semistable under $\Aut(X)$-action with respect to the linearization $\mO(c_1,...,c_k)$, we aim to show that $(X, \epsilon\sum_{j=1}^kc_jD_j)$ is K-semistable.
First, observe that there is a family $(\mX, \sum_{j=1}^k \mG_j)\to C$ over a smooth pointed curve $0\in C$ such that 
\begin{enumerate}
\item there is a morphism $C\setminus{0}\to \bP_j$ for each $j$;
\item $(\mX\setminus \mX_0,\mG_j\setminus \mG_{j,0})$ is obtained via pulling back $(X\times \bP_j,\mD_j)$ under the morphism $C\setminus \{0\}\to \bP_j$;
\item $(\mX_0, \sum_{j=1}^k\mG_{j,0})\cong (X, \sum_{j=1}^kD_j)$;
\item $(\mX_t, \epsilon\sum_{j=1}^kc_j\mG_{j,t})$
is K-semistable for any $t\in C\setminus \{0\}$.
\end{enumerate}
By the properness of K-moduli, up to a finite base change, one could replace the central fiber $(\mX_0, \epsilon\sum_{j=1}^kc_j\mG_{j,0})\cong (X, \epsilon\sum_{j=1}^kc_jD_j)$ with a K-semistable log Fano pair $(X', \epsilon\sum_{j=1}^kc_j D_j')$ and we denote the new family by $(\mX', \epsilon\sum_{j=1}^kc_j \mG_j')\to C$ for convenience. We claim that it suffices to show that $\mX'_0\cong X$. Suppose $\mX_0'\cong X$, then $(D_1',...,D_k')\in \bP_1\times ...\times \bP_k$ and it is GIT-semistable. By the separatedness of GIT moduli space, we know that $(D_1,...,D_k)$ and $(D_1',...,D_k')$ lie on the same $S$-equivalence class under $\Aut(X)$-action. Thus, $(X, \epsilon\sum_{j=1}^kc_jD_j)$ is K-semistable. 
Next, we turn to prove that $\mX_0'\cong X$. Since $\epsilon$ is sufficiently small, by Theorem \cite[Theorem 3.3]{Zhou21a}, we know that $\mX_0'$ is K-semistable. Note that $X$ is K-polystable, thus $\mX_0'\cong X$ by \cite[Theorem 3.4]{Zhou21a}. By now, we have finished the proof from GIT-semistability to K-semistability.
Finally, we suppose $(D_1,...,D_k)$ is GIT-polystable, then $(X, \epsilon\sum_{j=1}^kc_jD_j)$ is K-semistable as we have shown. To show that it is K-polystable, let $(X_0, \epsilon\sum_{j=1}^kc_jD_{j,0})$ be its K-polystable degeneration. By the same approach as above, we know that the ambient space of the K-polystable degeneration is preserved.
Thus 
$$(D_{1,0},...,D_{k,0})\in \bP_1\times...\times \bP_k$$ is GIT-polystable by the previous paragraph. By the separatedness of GIT-moduli space, we see $(D_1,...,D_k)\cong (D_{1,0},...,D_{k,0})$.
Therefore, $(X, \epsilon\sum_{j=1}^kc_jD_j)$ is itself K-polystable.
\end{proof}

Recall the following definition introduced in \cite{LZ23}.

\begin{definition} \label{def: kss domain}
Let $(D_1,...,D_k)\in \bP_1\times...\times \bP_k$, then the K-semistable domain of the log pair $(X, \sum_{j=1}^k D_j)$ is defined as follows:
$$\Kss(X, \sum_{j=1}^k D_j):=\overline{\{(x_1,...,x_k)\in [0,1]^k\cap \bQ^k\ |\ \textit{$(X, \sum_{j=1}^k x_jD_j)$ is K-semistable}\}}, $$
where the overline means taking the closure.
\end{definition}

By \cite[Theorem 1.4 and Corollary 1.5]{Zhou23b}, there exists a finite rational polytope chamber decomposition of $[0,1]^k$ such that for any $(D_1,...,D_k)\in \bP_1\times...\times \bP_k$, the K-semistability of $(X, \sum_{j=1}^kc_jD_j)$ does not change as $(c_1,...,c_k)$ varies in the interior domain of each chamber. We have

\begin{proposition}\label{prop: VGIT}
Near the original point of $[0,1]^k$, the above chamber decomposition exactly gives the chamber decomposition for the variation of GIT to encode the stability of GIT of $(D_1,...,D_k)\in \bP_1\times...\times \bP_k$ under the $\Aut(X)$-action with respect to the linearization $\mO(c_1,...,c_k)$.
\end{proposition}

\begin{proof}
The proof is a combination of Proposition \ref{prop: GIT=K} and \cite[Theorem 1.4 and Corollary 1.5]{Zhou23b}.
\end{proof}

\begin{proof}[Proof of Theorem \ref{thm: main1}]
The proof is a combination of Propositions \ref{prop: GIT=K} and \ref{prop: VGIT}.
\end{proof}

\section{K-moduli of Fano complete intersections}\label{sec: main2}

In this section, our goal is to prove Theorem \ref{thm: main2}. Throughout the section, we fix a positive integer $n$ and $k$ positive integers $\vec{d}:=(d_1,...,d_k)$ with $\sum_{j=1}^kd_j<n+1$. Denote by $\vec{a}:=(a_1,...,a_k)$ for
$$a_j=\frac{\sum_{i=1}^{k}d_i+(n-k+1)d_j-(n+1)}{(n-k+1)d_j}. $$
As noted in the introduction, we have $0\leq a_j<1$ for every $j$, and $a_j=0$ if and only if $d_1=d_2=...=d_k=1$. In the latter case, we easily see $a_1=a_2=...=a_k=0$ and $(\bP^n, \sum_{j=1}^ka_jS_j)=\bP^n$. We exclude this case in the following context since it is not interesting, therefore $0<a_j<1$ for every $j$.

\subsection{Inside $\bP^n$}
Let $\mM^K_{n, \vec{d}}$ (resp. $M^K_{n, \vec{d}}$) be the K-moduli stack (resp. K-moduli space) generically parametrizing $n-k$ dimensional Fano manifolds of the form $\cap_{j=1}^{k}S_j$ which are K-semistable, and $\mM^K_{n, \vec{d}, \vec{x}}$ (resp. $M^K_{n, \vec{d}, \vec{x}}$) the K-moduli stack (resp. K-moduli space) generically parametrizing K-semistable log Fano manifolds of the form  $(\bP^n, \sum_{j=1}^kx_jS_j)$, where $\vec{x}:=(x_1,...,x_k)$ and  $S_j$'s are hypersurfaces of degree $d_j$'s for $1\leq j\leq k$.

\begin{proposition}\label{prop: good S}
Suppose $(\bP^n, \sum_{j=1}^ka_jS_j)$ is a K-semistable log Fano pair (here $S_j$'s are not necessarily smooth), then every $S_j$ is irreducible and $S_j$'s are mutually distinct from each other.
\end{proposition}

\begin{proof}
By the choice of $a_j, j=1,...,k$, we know
$$1-a_i-S_{\bP^n, \sum_{j=1}^ka_jS_j}(S_i)=0 $$
for any $1\leq i\leq k$. We show that $S_1$ is irreducible. Suppose not, there is an irreducible component of $S_1$, denoted by $D$, satisfying
$$\beta_{\bP^n, \sum_{j=1}^ka_jS_j}(D)=1-a_1-S_{\bP^n, \sum_{j=1}^ka_jS_j}(D) <0.$$
The inequality holds since 
$$S_{\bP^n, \sum_{j=1}^ka_jS_j}(D)>S_{\bP^n, \sum_{j=1}^ka_jS_j}(S_1).$$
Thus, we get a contradiction to the K-semistability. Similarly, all other $S_j$'s are irreducible. Next we show that $S_j$'s are mutually different. Otherwise we may suppose $S_1=S_2$ up to reordering, then 
$$\beta_{\bP^n, \sum_{j=1}^ka_jS_j}(S_1)=1-(a_1+a_2)-S_{\bP^n, \sum_{j=1}^ka_jS_j}(S_1)<0, $$
which is again a contradiction to the K-semistability. The proof is finished.
\end{proof}

\begin{proposition}\label{prop: if}
Suppose $(\bP^n, \sum_{j=1}^ka_jS_j)$ is a K-semistable log Fano pair (here $S_j$'s are not necessarily smooth), then $\cap_{j=1}^k S_j$ is a K-semistable Fano variety of dimension $n-k$.
\end{proposition}

\begin{proof}
We divide the proof into several steps.

\

\textit{Step 1.}
Since $\beta_{\bP^n, \sum_{j=1}^ka_jS_j}(S_1)=0$, by Remark \ref{rem: kss degeneration}, one could produce a special test configuration $(\mX, \sum_{j=1}^k a_j\mG_j)\to \bA^1$ of $(\bP^n, \sum_{j=1}^{k}a_j S_j)$ such that the central fiber is the projective cone over $(S_1, \sum_{j=2}^ka_j{S_j}|_{S_1})$ with the polarization $\mO_{S_1}(d_1)$, and
$$\Fut(\mX, \sum_{j=1}^k a_j\mG_j; -K_{\mX/\bA^1})=0. $$
Thus, the central fiber $(\mX_0, \sum_{j=1}^k a_j\mG_{j,0})$ is also K-semistable. Note that we have
$$\mO_{S_1}(d_1)\sim_\bQ -\frac{1}{r}(K_{S_1}+\sum_{j=2}^ka_j{S_j}|_{S_1})$$
for
$$a_1=1-\frac{r}{n},\ \ \ r= \frac{n+1-d_1-\sum_{j=2}^ka_jd_j}{d_1}.$$ 
By Lemma \ref{lem:cone stability},  we see that $(S_1, \sum_{j=2}^ka_j{S_j}|_{S_1})$ is K-semistable. 

\

\textit{Step 2.} 
Applying induction, we assume 
$$(X_l, \sum_{j=l+1}^ka_jE^{(l)}_{j}):=(S_1\cap S_2\cap...\cap S_l, \sum_{j=l+1}^ka_j S_j\cap S_1\cap S_2\cap...\cap S_l)$$
is K-semistable. We aim to show that
$$(X_{l+1}, \sum_{j=l+2}^ka_jE^{(l+1)}_{j}):=(S_1\cap S_2\cap...\cap S_{l+1}, \sum_{j=l+2}^ka_j S_j\cap S_1\cap S_2\cap...\cap S_{l+1}) $$
is also K-semistable. First, observe that $(a_{l+1}, a_{l+2}, ...,a_k)$ is the unique solution of the following linear equations:
$$\beta_{X_l, \sum_{j=l+1}^kx_jE_j^{(l)}}(E_i^{(l)})=0,\ \ \ i=l+1,...,k. $$
By a similar proof of Proposition \ref{prop: good S}, one concludes that $E_j^{(l)}, j=l+1,...,k,$ are different from each other and every $E_j^{(l)}$ is irreducible. Note that
$$\beta_{X_l, \sum_{j=l+1}^ka_jE_j^{(l)}}(E_{l+1}^{(l)})=0, $$
thus similar to Step 1, there exists a special test configuration $(\mX', \sum_{j=l+1}^k a_j\mG'_j)\to \bA^1$ of $(X_l, \sum_{j=l+1}^ka_jE_j^{(l)})$ such that the central fiber is the projective cone over 
$$\left( X_{l+1}=E_{l+1}^{(l)}, \sum_{j=l+2}^ka_jE_j^{(l+1)}\right)=(S_1\cap S_2\cap...\cap S_{l+1}, \sum_{j=l+2}^ka_j S_j\cap S_1\cap S_2\cap...\cap S_{l+1})$$
with polarization $\mO_{X_{l+1}}(d_{l+1})$, and 
$$\Fut(\mX', \sum_{j=l+1}^ka_j\mG_j'; -K_{\mX'/\bA^1})=0. $$
Hence, the central fiber $(\mX'_0, \sum_{j=l+1}^k a_j\mG'_{j,0})$ is also K-semistable. Note that we have
$$\mO_{S_{l+1}\cap S_l\cap...\cap S_1}(d_{l+1})\sim_\bQ -\frac{1}{r}(K_{S_{l+1}\cap S_l\cap...\cap S_1}+\sum_{j=l+2}^ka_jS_j\cap S_{l+1}\cap S_l\cap...\cap S_1)$$
for
$$a_{l+1}=1-\frac{r}{n-l},\ \ \ r= \frac{n+1-\sum_{j=1}^{l+1}d_j-\sum_{j=l+2}^ka_jd_j}{d_{l+1}}.$$ 
By Lemma \ref{lem:cone stability}, we see that the following log pair is K-semistable:
$$(S_1\cap S_2\cap...\cap S_{l+1}, \sum_{j=l+2}^ka_j S_j\cap S_1\cap S_2\cap...\cap S_{l+1}).$$

\iffalse
By the similar proof of Proposition \ref{prop: good S}, one concludes that ${S_{d_j}}|_{S_{d_1}}, j=2,...,k,$ are different from each other and every ${S_{d_j}}|_{S_{d_1}}$ is irreducible. Note that
$$\beta_{S_{d_1}, \sum_{j=2}^ka_j{S_{d_j}}|_{S_{d_1}}}({S_{d_2}}|_{S_{d_1}})=0, $$
thus similar as Step 1, there exists a special test configuration $(\mX', \sum_{j=2}^k a_j\mG'_j)\to \bA^1$ of $(S_{d_1}, \sum_{j=2}^ka_j{S_{d_j}}|_{S_{d_1}})$ such that the central fiber is the projective cone over 
$$\left( {S_{d_2}}|_{S_{d_1}}, \sum_{j=3}^ka_j{({S_{d_j}}|_{S_{d_1}})}|_{({S_{d_2}}|_{S_{d_1}})}\right)=(S_{d_2}\cap S_{d_1}, \sum_{j=3}^k a_jS_{d_j}\cap S_{d_2}\cap S_{d_1})$$
with the polarization $\mO_{S_{d_1}\cap S_{d_2}}(d_2)$, and 
$$\Fut(\mX', \sum_{j=2}^ka_j\mG_j'; -K_{\mX'/\bA^1})=0. $$
Thus the central fiber $(\mX'_0, \sum_{j=2}^k a_j\mG'_{j,0})$ is also K-semistable. Note that we have
$$\mO_{S_{d_2}\cap S_{d_1}}(d_2)\sim_\bQ -\frac{1}{r}(K_{S_{d_2}\cap S_{d_1}}+\sum_{j=3}^ka_jS_{d_j}\cap S_{d_2}\cap S_{d_1})
\quad \text{and}  \quad
a_2=1-\frac{r}{n-1}$$
for 
$$r= \frac{n+1-d_1-d_2-\sum_{j=3}^ka_jd_j}{d_2}.$$ 
By Lemma \ref{lem:cone stability}, we see that $(S_{d_2}\cap S_{d_1}, \sum_{j=3}^k a_jS_{d_j}\cap S_{d_2}\cap S_{d_1})$ is K-semistable.
\fi

\

\textit{Step 3.} 
By Step 2, we see that the following log pair 
$$(S_{k-1}\cap...\cap S_1, a_kS_k\cap S_{k-1}\cap...\cap S_1)$$ 
is K-semistable and $S_{d_k}\cap S_{k-1}\cap...\cap S_1$ is irreducible. Since 
$$\beta_{S_{k-1}\cap...\cap S_1, a_kS_k\cap S_{k-1}\cap...\cap S_1}(S_k\cap S_{k-1}\cap...\cap S_1)=0,$$
there exists a special test configuration $(\mX'', a_k\mG_k'')\to \bA^1$ of 
$$(S_{k-1}\cap...\cap S_1, a_kS_k\cap S_{k-1}\cap...\cap S_1)$$ 
such that the central fiber is K-semistable and is the projective cone over $S_1\cap ...\cap S_k$ with polarization $\mO_{S_1\cap ...\cap S_k}(d_k)$. Moreover, we have
$$\mO_{S_k\cap...\cap S_1}(d_k)\sim_\bQ -\frac{1}{r}K_{S_k\cap...\cap S_1}$$
for
$$a_k=1-\frac{r}{n-k+1},\ \ \ r= \frac{n+1-\sum_{j=1}^{k}d_j}{d_k}.$$ 
Hence $S_1\cap ...\cap S_k$ is K-semistable by Lemma \ref{lem:cone stability}. The proof is finished.
\end{proof}

\begin{proposition}\label{prop: only if}
Given a log smooth pair $(\bP^n, \sum_{j=1}^kS_j)$, where $S_j$'s are distinct smooth hypersurfaces of degrees $d_j$ with $\sum_{j=1}^k d_j<n+1$. If $\cap_{j=1}^k S_j$ is a K-semistable  Fano variety of dimension $n-k$, then $(\bP^n, \sum_{j=1}^ka_jS_j)$ is a K-semistable log Fano pair.
\end{proposition}
\begin{proof}
In the log smooth setting, the proof of Proposition \ref{prop: if} can be reversed by Lemma \ref{lem:cone stability}. To see this, first note that 
$$(S_{k-1}\cap...\cap S_1, a_kS_k\cap S_{k-1}\cap...\cap S_1)$$
can be degenerated to a projective cone over $S_k\cap S_{k-1}\cap...\cap S_1$, denoted by $(\mX''_0, a_k\mG''_{k,0})$ (see the notation in Step 3 of the proof of Prop \ref{prop: if}). Since $(\mX''_0, a_k\mG''_{k,0})$ is K-semistable, we see 
$$(S_{k-1}\cap...\cap S_1, a_kS_k\cap S_{k-1}\cap...\cap S_1)$$
is also a K-semistable log Fano pair. Continuing to reverse the proof of Proposition \ref{prop: if}, we see
$$(S_1\cap S_2\cap...\cap S_l, \sum_{j=l+1}^ka_j S_j\cap S_1\cap S_2\cap ...\cap S_l)$$
is a K-semistable log Fano pair for any $1\leq l\leq k-1$. Finally, reversing Step 1 of the proof of Prop \ref{prop: if},  we see that $(\bP^n, \sum_{j=1}^ka_jS_j)$ can be degenerated to a projective cone over $(S_1, \sum_{j=1}^ka_jS_j\cap S_1)$, denoted by $(\mX_0, \sum_{j=1}^k a_j\mG_{j,0})$. Since $(\mX_0, \sum_{j=1}^k a_j\mG_{j,0})$ is a K-semistable log Fano pair, it is deduced that $(\bP^n, \sum_{j=1}^ka_jS_j)$ is itself a K-semistable log Fano pair.
\end{proof}

\begin{corollary}\label{cor: iff}
Given a log smooth pair $(\bP^n, \sum_{j=1}^kS_j)$, where $S_j$'s are distinct smooth hypersurfaces of degrees $d_j$ with $\sum_{j=1}^k d_j<n+1$. Then the following statements are equivalent:
\begin{enumerate}
\item $(\bP^n, \sum_{j=1}^ka_jS_j)$ is a K-semistable log Fano pair;
\item $S_1\cap S_2\cap...\cap S_k$ is a K-semistable Fano variety of dimension $n-k$;
\item $(S_1\cap S_2\cap...\cap S_l, \sum_{j=l+1}^ka_j S_j\cap S_1\cap S_2\cap ...\cap S_l)$ is a K-semistable log Fano pair for any $1\leq l\leq k$ (the log pair is identical to $S_1\cap S_2\cap...\cap S_k$ for $l=k$). 
\end{enumerate}
\end{corollary}

\begin{proof}
We can clearly see these equivalences from the proofs of Proposition \ref{prop: if} and Proposition \ref{prop: only if}.
\end{proof}

\begin{remark}\label{rem: kss to fano}
In Corollary \ref{cor: iff}, we put the log smooth condition to get these equivalent statements. However, in Proposition \ref{prop: if}, we do not need the log smooth assumption, and we see that $(\bP^n, \sum_{j=1}^ka_jS_j)$ being a K-semistable log Fano pair is a very strong condition, since it implies that $S_1, S_1\cap S_2,..., S_1\cap S_2\cap...\cap S_k$ are all Fano varieties; moreover, up to reordering $S_1,...,S_k$, we actually derive from the proof of Proposition \ref{prop: if} that the intersection $S_{i_1}\cap...\cap S_{i_p}$ is a Fano variety for any non-repeating subset $\{i_1,...,i_p\}\subset \{1,...,k\}$.
\end{remark}

\subsection{Outside $\bP^n$}
To study the K-moduli compactification, one has to consider the degeneration of $\bP^n$, which means that the ambient space can change.  We start with the following definition.

\begin{definition}\label{def: deg}
We say that the log pair $(X, \sum_{j=1}^kB_j)$ is a degeneration of log smooth pairs of the form $(\bP^n, \sum_{j=1}^kS_j)$ if there exists a $\bQ$-Gorenstein flat family over a smooth pointed curve $0\in C$, denoted by $(\mX, \sum_{j=1}^k\mB_j)\to C$, such that 
\begin{enumerate}
\item $\mX\to C$ is a flat family such that $-K_{\mX/C}$ is $\bQ$-Cartier and relatively ample;
\item each $\mB_j$ is a Weil divisor on $\mX$ and $\mB_j$ is proportional to $-K_{\mX/C}$ over $C$;
\item $(\mX_0, \sum_{j=1}^k\mB_{j,0})\cong (X, \sum_{j=1}^kB_j)$;
\item $(\mX_{t}, \sum_{j=1}^k\mB_{j, t})$ is of the form $(\bP^n, \sum_{j=1}^kS_j)$ for $t\in C\setminus \{0\}$, where $(\bP^n, \sum_{j=1}^kS_j)$ is log smooth and $S_j$'s are mutually distinct smooth hypersurfaces of degrees $d_j$'s.
\end{enumerate}
We say $(X, \sum_{j=1}^ka_jB_j)$ is a K-semistable degeneration if $(X, \sum_{j=1}^kB_j)$ is a degeneration as above and $(X, \sum_{j=1}^ka_jB_j)$ is a K-semistable log Fano pair. 

\end{definition}

\begin{lemma}\label{lem: smoothable deg}
Let $(\mX, \mB)\to C$ be a projective family of pairs over a smooth pointed curve $0\in C$ such that the family over $C\setminus 0$ is a log smooth, where $\mB$ is an effective prime divisor on $\mX$. Suppose (a) the degeneration $\mX_0$ has klt singularities;  (b) $-K_{\mX/C}$ is ample over $C$; (c) $\mB \sim_{\bQ} -rK_{\mX/C}$ for some rational $0<r<1$; (d) $\mB_0:=\mB|_{\mX_0}$ is normal. Then we have
\begin{enumerate}
\item $\mX\to C$ is a $\bQ$-Gorenstein family of Fano varieties;
\item if $\mB_0$ has klt type singularities, then $\mB\to C$ is a $\bQ$-Gorenstein family of Fano varieties;
\item $(K_{\mX_0}+\mB_0)|_{\mB_0}=K_{\mB_0}$, i.e. there is no different part.
\end{enumerate}
\end{lemma}

\begin{proof}
Note that the family $(\mX, \mB)\to C$ is log smooth over $C\setminus 0$, we easily see that $\mX\to C$ is a $\bQ$-Gorenstein family of Fano varieties by conditions (a) and (b). Since $\mX$ is smooth in codim 2, we have the following adjunction formulas
$$ {\rm (\star)}\ \ (K_{\mX}+\mB)|_{\mB}=K_{\mB}, \ \ \ {\rm (\star\star)}\ \  (K_\mB+\mB_0)|_{\mB_0}=K_{\mB_0}. $$
By condition (c), we see that $-K_\mB$ is ample over $C$. Combining condition (d), we see that $\mB\to C$ is also a $\bQ$-Gorenstein family of Fano varieties if $\mB_0$ has klt type singularities, which concludes (2). By adjunction, we have
$$(K_{\mX_0}+\mB_0)|_{\mB_0}=K_{\mB_0}+\Diff_{\mB_0}(0). $$
We aim to show that $\Diff_{\mB_0}(0)=0$ which concludes (3). To see this, first note that $\Diff_{\mB_0}(0)\geq 0$ (e.g. \cite[Proposition 4.5]{Kollar13}). Observe the following:
$$-(K_{\mX_0}+\mB_0)=-(K_\mX+\mX_0+\mB)|_{\mX_0}\sim_{\bQ, C} -(1-r)K_{\mX_0}, $$
which implies the following:
$$-(K_{\mB_0}+\Diff_{\mB_0}(0))\sim_\bQ -(1-r)K_{\mX_0}|_{\mB_0}.$$
On the other hand, by formulas $(\star)$ and $(\star\star)$, we have
$$-K_{\mB_0}=-(K_{\mB}+\mB_0)|_{\mB_0}=-(K_\mX+\mB)|_{\mB_0}\sim_\bQ -(1-r)K_{\mX_0}|_{\mB_0}. $$
Thus, we derive $-(K_{\mB_0}+\Diff_{\mB_0}(0))\sim_\bQ -K_{\mB_0}$ and it is clear that $\Diff_{\mB_0}(0)=0$.
\end{proof}

\begin{proposition}\label{prop: deg}
Notation as in Definition \ref{def: deg}, let $(X, \sum_{j=1}^ka_jB_j)$ be a K-semistable degeneration. Denoting by $B_0:=X$ and $B^{(l)}:=B_0\cap B_1\cap...\cap B_l$, then we have the following conclusions:
\begin{enumerate}
\item $B^{(l)}$ is a Fano variety for any $0\leq l\leq k$;
\item $K_{B^{(l)}}+B^{(l+1)}=K_{B^{(l+1)}}$ for any $0\leq l\leq k-1$, i.e. there is no different part;
\item $B^{(k)}=B_1\cap B_2\cap...\cap B_k$ is a K-semistable Fano variety. 
\end{enumerate}
\end{proposition}

\begin{proof}
Let $(\mX, \sum_{j=1}^ka_j\mB_j)\to C$ be the degeneration family over a pointed curve $0\in C$ as in Definition \ref{def: deg}, for which we have $(\mX_0, \sum_{j=1}^ka_j\mB_{j,0})\cong (X, \sum_{j=1}^ka_jB_j)$. Note that the degenerating family is log smooth over $C\setminus 0$, and each $\mB_j$ is a prime divisor on $\mX$ satisfying $\mB_j\sim_{\bQ, C}-r_jK_{\mX/C}$ for $r_j=\frac{d_j}{n+1}<1$.

\

\textit{Step 1.} 
Since  $(X, \sum_{j=1}^k a_j B_j)$ is K-semistable, then the Proposition \ref{prop: good S} stated for $(\bP^n, \sum_{j=1}^ka_jS_j)$ also applies to $(X, \sum_{j=1}^ka_jB_j)$. Indeed, we observe that $(a_1,...,a_k)$ is still the unique solution of the following linear equations
$$1-a_i-S_{X, \sum_{j=1}^ka_jB_j}(B_i)=0,\ \ \ i=1,2,...,k, $$
since we still have the following for a general fiber $(\bP^n, \sum_{j=1}^ka_jS_j)$:
$$S_{X, \sum_{j=1}^ka_jB_j}(B_i)=S_{\bP^n, \sum_{j=1}^ka_jS_j}(S_i).$$
Thus it is similar for $(X, \sum_{j=1}^kB_j)$ to derive that $B_j$'s are irreducible and mutually different. Since $\beta_{X, \sum_{j=1}^ka_jB_j}(B_1)=0$, by \cite{BLZ22}, we know that $B_1$ induces a special test configuration of $(X, \sum_{j=1}^ka_jB_j)$, denoted by
$$(\mX', \sum_{j=1}^ka_j\mB'_j)\to \bA^1,$$ 
where the central fiber is also a K-semistable log Fano pair and $(\mX'_0, \sum_{j=1}^ka_j\mB'_{j,0})$ is the projective orbifold cone over $B_1$. Hence $B_1$ is a klt type divisor on $X$ (e.g. \cite[Section 3.1]{Kollar13}). Applying Lemma \ref{lem: smoothable deg} to $(\mX, \mB_1)\to C$, we see that $B_1\cong \mB_{1,0}$ is a Fano variety, and $(K_X+B_1)|_{B_1}=K_{B_1}$. Since there is no different part, we see that $X$ is smooth along the codimension two points of $B_1$, which implies that $B_j|_{B_1}=B_j\cap B_1$ are all prime divisors on $B_1$ for $2\leq j\leq k$. By the orbifold cone construction (e.g. \cite[Section 4]{Kollar04seifert} or \cite[section 3.1]{Zhuang24}), we see that $(\mX'_0, \sum_{j=1}^ka_j\mB'_{j,0})$ is the projective cone over $(B_1, \sum_{j=2}^k a_j B_j\cap B_1)$. By Lemma \ref{lem:cone stability}, we see that $(B_1, \sum_{j=2}^ka_j B_j\cap B_1)$ is a K-semistable log Fano pair.

\

\textit{Step 2.} 
Now, we consider the degenerating family $(\mB_1, \sum_{j=2}^k a_j \mB_j\cap \mB_1)\to C$, which is also a family of K-semistable log Fano pairs. Similarly as in Step 1 (see Step 2 of the proof of Proposition \ref{prop: if}), we see that $\mB_{j,0}\cap \mB_{1,0}=B_j\cap B_1$ are mutually different prime divisors on $B_1$ for $2\leq j\leq k$, and $B_2\cap B_1$ induces a special test configuration of $(B_1, \sum_{j=2}^ka_j B_j\cap B_1)$ with vanishing generalized Futaki invariant. Using the same argument as in Step 1, we conclude that $B_2\cap B_1$ is a Fano variety and $K_{B_1}+B_2\cap B_1=K_{B_2\cap B_1}$. Continuing the process, we clearly complete the proofs of (1) and (2). Applying Lemma \ref{lem:cone stability} in the final step, we conclude (3).  
\end{proof}

\begin{remark}\label{rem: kss to fano'}
In Proposition \ref{prop: deg}, similar to Remark \ref{rem: kss to fano}, we actually derive that the intersection $B_{i_1}\cap...\cap B_{i_p}$ is a Fano variety for any non-repeating subset $\{i_1,...,i_p\}\subset \{1,...,k\}$.
\end{remark}

\subsection{Connecting map}

Recall that $\mM^K_{n, \vec{d}}$ (resp. $M^K_{n, \vec{d}}$) is the K-moduli stack (resp. K-moduli space) generically parametrizing $n-k$ dimensional Fano manifolds of the form $\cap_{j=1}^{k}S_j$ which are K-semistable, and $\mM^K_{n, \vec{d}, \vec{a}}$ (resp. $M^K_{n, \vec{d}, \vec{a}}$) is the K-moduli stack (resp. K-moduli space) generically parametrizing K-semistable log Fano manifolds of the form  $(\bP^n, \sum_{j=1}^ka_jS_j)$, where $\vec{a}:=(a_1,...,a_k)$ and  $S_j$'s are hypersurfaces of degree $d_j$'s for $1\leq j\leq k$. We have seen that $\mM^K_{n,\vec{d}}$ and $\mM^K_{n, \vec{d},\vec{a}}$ are closely related to each other by Corollary \ref{cor: iff} and Proposition \ref{prop: deg}. We will establish the connection and complete the proof of Theorem \ref{thm: main2} in this subsection.

\begin{proof}[Proof of Theorem \ref{thm: main2}]
By Corollary \ref{cor: iff}, we directly conclude the first statement in Theorem \ref{thm: main2}. We divide the rest proof into two steps. First, we construct the connecting map $\mM^K_{n, \vec{d},\vec{a}}\to \mM^K_{n,\vec{d}}$. Then we show that it descends to a surjective map between moduli spaces. 

\

\textit{Step 1.}
By Proposition \ref{prop: deg}, given $(X, \sum_{j=1}^ka_jB_j)\in \mM^K_{n, \vec{d}, \vec{a}}$, we have that $\cap_{j=1}^k B_j$ is also K-semistable of dimension $n-k$. Thus we can define the morphism
$$\mM^K_{n, \vec{d}, \vec{a}}\to \mM^K_{n, \vec{d}},\ \  [(X, \sum_{j=1}^ka_jB_j)]\mapsto [\cap_{j=1}^k B_j].$$
Actually, the above morphism is only defined on $\bC$-points, however, we could extend it to be a morphism between stacks as follows. For convenience, we only define the morphism for $\mM^K_{n, \vec{d}, \vec{a}}(W)\to \mM^K_{n, \vec{d}}(W)$, where $W$ is a normal scheme\footnote{For arbitrary base, we take the concept of K-flatness introduced by \cite{Kollar19families}. Refer to \cite[Section 2.6]{XZ20b} for the definition of the family.}. Take $[(\mX, \sum_{j=1}^ka_j\mB_j)\to W]\in \mM^K_{n, \vec{d}, \vec{a}}(W)$, i.e.
\begin{enumerate}
\item $\mX\to W$ is a flat family such that $-K_{\mX/W}$ is $\bQ$-Cartier and ample over $W$;
\item $\mB_j$ is a prime divisor on $\mX$ with $\mB_j\sim_{\bQ, W} -r_jK_{\mX/W}$ and $r_j=\frac{d_j}{n+1}$;
\item $(\mX_t, \sum_{j=1}^ka_j\mB_{j,t})\in \mM^K_{n, \vec{d}, \vec{a}}(\bC)$ for every closed point $t\in W$.
\end{enumerate}
By Proposition \ref{prop: deg} and Remark \ref{rem: kss to fano'}, we see that $\mB_{i_1,t}\cap \mB_{i_2,t}\cap...\cap \mB_{i_p,t}$ is a Fano variety for any closed point $t\in W$ and non-repeating subset $\{i_1,...,i_p\}\subset\{1,2,...,k\}$, thus $\mB_{i_1}\cap \mB_{i_2}\cap...\cap \mB_{i_p}\to W$ is flat by \cite[Lemma 10.58]{Kollar23}. By Proposition \ref{prop: deg} (2), we see that the family $\mB_1\cap \mB_2\cap...\cap \mB_k\to W$ could be obtained by taking $k$ times adjunction, and every time we take the adjunction, there is no difference part.  
By Proposition \ref{prop: deg} (3), we see that $\mB_1\cap \mB_2\cap...\cap \mB_k\to W$ is a $\bQ$-Gorenstein family of K-semistable Fano varieties. Thus the connecting morphism could be defined by sending $[(\mX, \sum_{j=1}^ka_j\mB_j)\to W]$ to $[\mB_1\cap \mB_2\cap...\cap \mB_k\to W]$.

\

\textit{Step 2.}
The morphism between stacks naturally descends to a morphism between K-moduli spaces by the universal property of good moduli spaces (e.g. \cite[Theorem 4.16 (vi)]{Alper13}):
$$\phi: M^K_{n, \vec{d}, \vec{a}}\to M^K_{n, \vec{d}}. $$
We aim to show that $\phi$ is surjective. Suppose $[Y]\in M^K_{n, \vec{d}}$  is of the form $S_1\cap...\cap S_k$ deduced by a log smooth pair $(\bP^n, \sum_{j=1}^k S_j)$,  by 
Corollary \ref{cor: iff}, $[(\bP^n, \sum_{j=1}^ka_jS_j)]$ is a preimage. Suppose $[Y]$ is not of the form $S_1\cap...\cap S_k$ deduced by a log smooth pair $(\bP^n, \sum_{j=1}^k S_j)$. Then there is a flat family of K-semistable Fano varieties $\mY\to C$ over a smooth pointed curve $0\in C$ such that $-K_{\mY/C}$ is ample over $C$, $\mY_0\cong Y$, and $\mY_t$ is of the form $S_1\cap...\cap S_k$ deduced by a log smooth pair $(\bP^n, \sum_{j=1}^k S_j)$. Let $\mY^*\to C^*:=C\setminus \{0\}$ be the family obtained via base change under the inclusion $C^*\to C$. We claim that there exist an \'etale morphism $C'^*\to C^*$ and a family 
$$(\mX, \sum_{j=1}^k\mS_j)\to C'^*$$
such that for any $t\in C'^*$, the fiber $(\mX_t, \sum_{j=1}^k\mS_{j, t})$ is isomorphic to some log smooth pair $(\bP^n, \sum_{j=1}^k S_j)$ and $\cap_{j=1}^k \mS_{j, t}\cong \mY'^*_t$, where $\mY'^*\to C'^*$ is the pulling back of $\mY^*\to C^*$ along $C'^*\to C^*$. To see this, we denote by $\bP_j:=|\mO_{\bP^n}(d_j)|$ and $\mE_j\subset \bP^n\times \bP_j$ the universal divisor.
Putting $\mD_j:= \bP_1\times...\bP_{j-1}\times\mE_j\times \bP_{j+1}\times ...\times \bP_k$ and $T:=\bP_1\times \bP_2\times...\times \bP_k$,
we consider the following universal family
$$(\bP^n\times T, \sum_{j=1}^k \mD_j)\to T. $$
Let $Z\subset T$ be the subset of $T$ parametrizing those log smooth fibers $(\bP^n, \sum_{j=1}^{k}\mD_{j,t})$ where $\mD_{j,t}, j=1,...,k,$ are mutually different. We denote the family by 
$$(\bP^n\times Z, \sum_{j=1}^k\mD_j\times_T Z)\to Z.$$
%It is clear that $\bigcap_{j=1}^k\ \mD_j\times_T Z \to Z$ parametrizes all complete intersections of $k$ different hypersurfaces in $\bP^n$ of degrees $d_1, d_2,...,d_k$ respectively.  
Let $\mU\to U$ be the universal family of complete intersections of $k$ different hypersurfaces in $\bP^n$ of degrees $d_1, d_2,...,d_k$ respectively, then we naturally have the following commutative diagram

\begin{center}
	\begin{tikzcd}[column sep = 2em, row sep = 2em]
	 (\bP^n\times Z, \sum_{j=1}^k\mD_j\times_T Z) \arrow[d,"", swap] \arrow[rr,""]&& \bigcap_{j=1}^k\ \mD_j\times_T Z \arrow[d,"",swap] \arrow[rr,""]& & \mU \arrow[d,"", swap]  \\
	 Z\arrow[rr,"="]&& Z \arrow[rr,""]	&& U
	 \end{tikzcd}
\end{center}
Note that there is a morphism $C^*\to U$ such that $\mY^*\to C^*$ is the pulling back of $\mU\to U$ via the base change $C^*\to U$. Taking base change again, we obtain a family of log smooth pairs
$$\pi: (\bP^n\times Z\times_U C^*, \sum_{j=1}^k\mD_j\times_T Z\times_U C^*)\to Z\times_U C^*. $$ 
Up to shrinking $C$ around $0\in C$, one could find a section $C'^*\to Z\times_U C^*$ such that $C'^*\to C^*$ is an \'etale morphism. Pulling back the family $\pi$ via this section, we get the required family $(\mX, \sum_{j=1}^k\mS_j)\to C'^*$.

Note that $(\mX, \sum_{j=1}^ka_j\mS_j)\to C'^*$ is a family of K-semistable log Fano pairs, by the properness of K-moduli space, we could find a complete family $$(\tilde{\mX}, \sum_{j=1}^ka_j\tilde{\mS}_j)\to C'\ni 0'$$ 
up to a finite base change and the central fiber $(\tilde{\mX}_{0'}, \sum_{j=1}^ka_j\tilde{\mS}_{j,{0'}})$ is also K-semistable. By Proposition \ref{prop: if},  
$\cap_{j=1}^k \tilde{\mS}_{j,{0'}}$ is K-semistable and $[Y]=[\cap_{j=1}^k \tilde{\mS}_{j,{0'}}]$ by the separatedness of K-moduli space (e.g. \cite{BX19}). Thus $[(\tilde{\mX}_{0'}, \sum_{j=1}^ka_j\tilde{\mS}_{j,{0'}})]$ is a preimage of $[Y]$ and $\phi$ is surjective. The proof is complete.
\end{proof}

\begin{example}
Let $n, d$ be two positive integers with $d<n+1$ and put $a:=1-\frac{n+1-d}{nd}$, then there is a natural morphism $\mM^K_{n,d,a}\to \mM^K_{n,d}$ which descends to a surjective morphism $M^K_{n,d,a}\to M^K_{n,d}$. In particular, if $d=1$, then $a=0$. In this case, $\mM^K_{n,d,a}$ only parametrizes $\bP^n$ and $\mM^K_{n, d}$ only parametrizes $\bP^{n-1}$.
\end{example}

\begin{example}\label{exa: cub-hyp}
Take $n=4$ and $\vec{d}=(d_1,d_2)=(3,1)$. Denote $\vec{a}:=(a_1,a_2)=(\frac{8}{9}, \frac{2}{3})$. Then there is a morphism $\mM^K_{4,\vec{d}, \vec{a}}\to \mM^K_{4, \vec{d}}$ sending a K-semistable log Fano pair $(\bP^4, \frac{8}{9}S+\frac{2}{3}H)$ to a K-semistable cubic surface $S\cap H$, where $S$ is a cubic 3-fold. This morphism descends to a surjective morphism $M^K_{4,\vec{d}, \vec{a}}\to M^K_{4, \vec{d}}$. We also note that $M^K_{4, \vec{d}}$ is isomorphic to the GIT moduli space of cubic surfaces.
\end{example}

\begin{example}
Take $n=4$ and $\vec{d}=(d_1,d_2)=(2,2)$. Denote $\vec{a}:=(a_1,a_2)=(\frac{5}{6}, \frac{5}{6})$. Then there is a morphism $\mM^K_{4,\vec{d}, \vec{a}}\to \mM^K_{4, \vec{d}}$ sending a K-semistable log Fano pair $(\bP^4, \frac{5}{6}Q_1+\frac{5}{6}Q_2)$ to a K-semistable del Pezzo surface $Q_1\cap Q_2$ of degree $4$. This morphism descends to a surjective morphism $M^K_{4,\vec{d}, \vec{a}}\to M^K_{4, \vec{d}}$. We also note that $M^K_{4, \vec{d}}$ is isomorphic to the GIT moduli space of smoothable del Pezzo surfaces of degree 4 under the $\PGL(5)$-action with respect to the polarization induced by the Pl\"ucker embedding (e.g. \cite{SS17, OSS16, MM93}).
\end{example}

\begin{remark}
Fix a positive integer $n$ and $k$ positive integers $\vec{d}:=(d_1,...,d_k)$. Let $\vec{a}:=(a_1,...,a_k)$ be the vector defined at the beginning of this section. 
Denote $\mM^{\GIT}_{n,\vec{d},\vec{a}}$ (resp. $M^{\GIT}_{n,\vec{d},\vec{a}}$) as the GIT-moduli stack (resp. GIT-moduli space) parameterizing GIT-semistable (resp. GIT-polystable) elements
$$(S_1,...,S_k)\in |\mO_{\bP^n}(d_1)|\times...\times 
|\mO_{\bP^n}(d_k)|$$ 
under $\Aut(\bP^n)$-action with respect to the linearization $\mO(a_1,...,a_k)$. By  Lemma \ref{lem: proportional cm} and Proposition \ref{prop: GIT=K}, we have 
$$\mM^{\GIT}_{n,\vec{d},\vec{a}}\cong \mM^{K}_{n,\vec{d},c\vec{a}}\ \ ({\rm{resp}}.\  M^{\GIT}_{n,\vec{d},\vec{a}}\cong M^{K}_{n,\vec{d},c\vec{a}})$$
for rational $0<c\ll 1$. As $c$ varies in $(0,1]\cap \bQ$, there is a wall crossing theory with one boundary (\cite{ADL19, Zhou23}). When $c=1$, applying Theorem \ref{thm: main2},  there is a surjective morphism $M^{K}_{n, \vec{d}, \vec{a}}\to M^K_{n, \vec{d}}$. Hence, Theorem \ref{thm: main1} together with Theorem \ref{thm: main2} actually gives a relationship between $M^{\GIT}_{n, \vec{d}, \vec{a}}$ and $M^K_{n, \vec{d}}$. More precisely, we first have a birational map $M^{\GIT}_{n, \vec{d}, \vec{a}}\dashrightarrow M^{K}_{n, \vec{d}, \vec{a}}$ via the wall crossing with one boundary, then we have a projective surjective morphism $M^{K}_{n, \vec{d}, \vec{a}}\to M^{K}_{n, \vec{d}}$.
\end{remark}

\section{An example: Cubic-Hyperplane pairs in $\bP^4$}\label{sec: cubic-hyperplane}

In this section, we consider the pair of the form $(\bP^4, aS+bH)$, where $S$ is a cubic hypersurface, $H$ is a hyperplane, and $a, b\in [0,1]$. Denote by $G:=\PGL(5)$. It is natural to consider the GIT-moduli of $\bP_1\times \bP_2$ under the $G$-action with respect to the linearization $\mO(a, b), \ a>0,b>0$, where  $\bP_1:=|\mO_{\bP^4}(3)|$ and $\bP_2:=|\mO_{\bP^4}(1)|$. Denote by $t:=\frac{b}{a}$ and let $M^{\GIT}(t)$ be the $\GIT$-moduli space with respect to the linearization $\mO(a, b)$ for $t=\frac{b}{a}$. Then with the help of a computer program by {\rm (\cite{GM})} we have the following result.

\begin{proposition}\label{prop: VGIT}
The variation of $M^{\GIT}(t)$ is controlled by the following $\GIT$-walls 
$$t\in \bigg\{\frac{1}{13}, \frac{1}{8},\frac{3}{19},\frac{2}{11},\frac{3}{14},\frac{6}{23}, \frac{1}{3},\frac{3}{7},\frac{6}{13},\frac{11}{23},\frac{1}{2},\frac{9}{17},\frac{4}{7}, \frac{7}{11},\frac{3}{4}\bigg\}. $$
Moreover, $M^{\GIT}(\frac{3}{4})$ is a point, and for $t\in (\frac{3}{4},\infty)$, the space of the GIT modules $M^{\GIT}(t)$ is empty.
\end{proposition}

Recall that Theorem \ref{thm: main1} establishes the relationship between the variation of GIT-stability of $[S, H]\in \bP_1\times \bP_2$ with respect to $t:=\frac{b}{a}$ and the variation of K-stability of the log Fano pair $(\bP^4, \epsilon(aS+bH))$, where $0<\epsilon\ll 1$ is sufficiently small. That is, $(S, H)\in \bP_1\times \bP_2$ is GIT-semistable with respect to the linearization $\mO(a, b)$ if and only if $(\bP^4, \epsilon(aS+bH))$ is K-semistable for $0<\epsilon\ll 1$. On the other hand, by Theorem \ref{thm: main2} (see also Example \ref{exa: cub-hyp}) we know that $(\bP^4,\frac89S+\frac23H)$ is K-semistable if and only if the cubic surface $S\cap H$ is GIT semistable. 

For a log smooth pair $(\bP^4, S+H)$, its K-semistable domain $\Kss(\bP^4, S+H)$ is determined by three vertices $(a, b)=\{(0,0), (\frac{5}{6}, 0), (\frac{8}{9},\frac{2}{3}) \}$ (e.g. Theorem \ref{thm: domain}). We denote this polytope by $P$. The slopes determining the variation of $\GIT$-moduli spaces in Proposition \ref{prop: VGIT} are presented as follows (we only mark part of the slopes to avoid bloating).

\begin{center}
\begin{tikzpicture}[
    scale=5,
    >=latex,
    every node/.style={font=\small}
]

% ===== 坐标轴 =====
\draw[->,thin] (0,0) -- (1.45,0) node[right] {$a$};
\draw[->,thin] (0,0) -- (0,1.25) node[above] {$b$};

% ===== 横轴9等分刻度 =====
\foreach \x in {1,...,9}
  \draw[thin] (\x/9,0) -- (\x/9,-0.015);

% 横轴标签
\node[below] at (5/6,0) {$\frac56$};
\node[below] at (8/9,0) {$\frac89$};
\node[below] at (1,0) {$1$};

% ===== 纵轴9等分刻度 =====
\foreach \y in {1,...,9}
  \draw[thin] (0,\y/9) -- (-0.015,\y/9);

% 纵轴标签
\node[left] at (0,2/3) {$\frac23$};
\node[left] at (0,1) {$1$};

% ===== 三角形 =====
\draw[thick] (0,0) -- (5/6,0) -- (8/9,2/3) -- cycle;

% ===== 点 (8/9,2/3) =====
\fill (8/9,2/3) circle (0.5pt);
\node[above right] at (8/9,2/3)
{$\left(\frac89,\frac23\right)$};

% ===== 射线参数 =====
\def\xend{1.35}    % 射线延长长度
\def\xsolid{0.28}  % 原点附近实线长度

% ===== slope 1/13 =====
\draw[thick] (0,0) -- (\xsolid,\xsolid/13);
\draw[thick,dashed] (\xsolid,\xsolid/13) -- (\xend,\xend/13);
\node[right] at (\xend,\xend/13) {slope $\frac{1}{13}$};

% ===== slope 1/8 =====
\draw[thick] (0,0) -- (\xsolid,\xsolid/8);
\draw[thick,dashed] (\xsolid,\xsolid/8) -- (\xend,\xend/8);
\node[right] at (\xend,\xend/8+0.02) {slope $\frac{1}{8}$};

% ===== slope 4/7 =====
\draw[thick] (0,0) -- (\xsolid,\xsolid*4/7);
\draw[thick,dashed] (\xsolid,\xsolid*4/7) -- (\xend,\xend*4/7);
\node[right] at (\xend,\xend*4/7+0.06) {slope $\frac{4}{7}$};

% ===== slope 7/11 =====
\draw[thick] (0,0) -- (\xsolid,\xsolid*7/11);
\draw[thick,dashed] (\xsolid,\xsolid*7/11) -- (\xend,\xend*7/11);
\node[right] at (\xend,\xend*7/11+0.08) {slope $\frac{7}{11}$};

% ===== slope 3/4 =====
\draw[thick] (0,0) -- (\xsolid,\xsolid*3/4);
\draw[thick, dashed] (\xsolid,\xsolid*3/4) -- (\xend,\xend*3/4);
\node[right] at (\xend,\xend*3/4+0.08) {slope $\frac{3}{4}$};

\end{tikzpicture}
\end{center}

By \cite{Zhou23b}, there exists a finite rational polytope chamber decomposition of $P$ to control the variation of K-semistability of any $(X, aB_1+bB_2)$, where $(X, B_1+B_2)$ arises from the degeneration of $(\bP^4, S+H)$ (see Definition \ref{def: deg}). Using Theorem \ref{thm: main1}, we can see in the above picture what the chamber decomposition looks like near the original point $(0,0)$. Determining the complete chamber decomposition of $P$ seems a great challenge. However, it is possible to explore the walls on the segment starting from $(0,0)$. For example, we have the following result. 

\begin{proposition}
There are no walls inside the segment $[(0,0), (\frac{5}{6},0)]$.
\end{proposition}

\begin{proof}
Consider the log pair $(\bP^4, c\cdot \frac{5}{6}S)$, where $S$ is a cubic threefold and $c\in (0,1)$. We show that the following statements are equivalent:
\begin{enumerate}
\item $(\bP^4, c\cdot \frac{5}{6}S)$ is K-semistable for $0<c\ll 1$;
\item $S$ is $\GIT$-semistable;
\item $S$ is K-semistable;
\item $(\bP^4, \frac{5}{6}S)$ is K-semistable;
\item $(\bP^4, c\cdot\frac{5}{6}S)$ is K-semistable for any $c\in (0,1)$.
\end{enumerate}
Note that (1) implies (2) by \cite[Theorem 1.4]{ADL19}; (2) implies (3) by \cite{LX19}. To see how (3) implies (4), let $(X, \frac{5}{6}B)$ denote the projective cone over $S$ with respect to the polarization $S|_S$, where $B$ is the infinite section of the projective cone $X$. By Lemma \ref{lem:cone stability}, we see that $(X, \frac{5}{6}B)$ is K-semistable. Note that $(X, \frac{5}{6}B)$ is a degeneration of $(\bP^4, \frac{5}{6}S)$ by a special test configuration, so we get $(\bP^4, \frac{5}{6}S)$ is K-semistable by \cite{BLX22}. (4) implies (5) by the interpolation property of K-stability and (5) implies (1) clearly.

These equivalences already imply that there is no wall for $c\in (0,1)$. Indeed, suppose $c_1\in (0,1)$ is the first wall, then there exists a K-semistable log Fano pair $(\bP^4, c_1\cdot \frac{5}{6}S)$ such that $(\bP^4, c\cdot \frac{5}{6}S)$ is not K-semistable for $c>c_1$, which is a contradiction. 
\end{proof}

We end this article with the following question. 

\begin{question}
Whether there is any wall inside the segment $[(0,0), (\frac{8}{9},\frac{2}{3})]$?
\end{question}

This question will be discussed in a forthcoming work.

\bibliography{reference.bib}

\begin{bibdiv}
\begin{biblist}

\bib{ABHLX20}{article}{
      author={Alper, Jarod},
      author={Blum, Harold},
      author={Halpern-Leistner, Daniel},
      author={Xu, Chenyang},
       title={Reductivity of the automorphism group of {K}-polystable {F}ano
  varieties},
        date={2020},
        ISSN={0020-9910},
     journal={Invent. Math.},
      volume={222},
      number={3},
       pages={995\ndash 1032},
         url={https://doi.org/10.1007/s00222-020-00987-2},
      review={\MR{4169054}},
}

\bib{ADL20}{article}{
      author={Ascher, Kenneth},
      author={DeVleming, Kristin},
      author={Liu, Yuchen},
       title={K-moduli of curves on a quadric surface and {K}3 surfaces},
        date={2023},
        ISSN={1474-7480},
     journal={J. Inst. Math. Jussieu},
      volume={22},
      number={3},
       pages={1251\ndash 1291},
         url={https://doi.org/10.1017/S1474748021000384},
      review={\MR{4574172}},
}

\bib{ADL21}{article}{
      author={Ascher, Kenneth},
      author={DeVleming, Kristin},
      author={Liu, Yuchen},
       title={K-stability and birational models of moduli of quartic {K}3
  surfaces},
        date={2023},
        ISSN={0020-9910},
     journal={Invent. Math.},
      volume={232},
      number={2},
       pages={471\ndash 552},
         url={https://doi.org/10.1007/s00222-022-01170-5},
      review={\MR{4574660}},
}

\bib{ADL19}{article}{
      author={Ascher, Kenneth},
      author={DeVleming, Kristin},
      author={Liu, Yuchen},
       title={Wall crossing for {K}-moduli spaces of plane curves},
        date={2024},
        ISSN={0024-6115},
     journal={Proc. Lond. Math. Soc. (3)},
      volume={128},
      number={6},
       pages={Paper No. e12615, 113},
         url={https://doi.org/10.1112/plms.12615},
      review={\MR{4757702}},
}

\bib{Alper13}{article}{
      author={Alper, Jarod},
       title={Good moduli spaces for {A}rtin stacks},
        date={2013},
        ISSN={0373-0956},
     journal={Ann. Inst. Fourier (Grenoble)},
      volume={63},
      number={6},
       pages={2349\ndash 2402},
         url={http://aif.cedram.org/item?id=AIF_2013__63_6_2349_0},
      review={\MR{3237451}},
}

\bib{BLX22}{article}{
      author={Blum, Harold},
      author={Liu, Yuchen},
      author={Xu, Chenyang},
       title={Openness of {K}-semistability for {F}ano varieties},
        date={2022},
        ISSN={0012-7094},
     journal={Duke Math. J.},
      volume={171},
      number={13},
       pages={2753\ndash 2797},
         url={https://doi.org/10.1215/00127094-2022-0054},
      review={\MR{4505846}},
}

\bib{BLZ22}{article}{
      author={Blum, Harold},
      author={Liu, Yuchen},
      author={Zhou, Chuyu},
       title={Optimal destabilization of {K}-unstable {F}ano varieties via
  stability thresholds},
        date={2022},
        ISSN={1465-3060},
     journal={Geom. Topol.},
      volume={26},
      number={6},
       pages={2507\ndash 2564},
         url={https://doi.org/10.2140/gt.2022.26.2507},
      review={\MR{4521248}},
}

\bib{BX19}{article}{
      author={Blum, Harold},
      author={Xu, Chenyang},
       title={Uniqueness of {${K}$}-polystable degenerations of {F}ano
  varieties},
        date={2019},
        ISSN={0003-486X},
     journal={Ann. of Math. (2)},
      volume={190},
      number={2},
       pages={609\ndash 656},
         url={https://doi.org/10.4007/annals.2019.190.2.4},
      review={\MR{3997130}},
}

\bib{CP21}{article}{
      author={Codogni, Giulio},
      author={Patakfalvi, Zsolt},
       title={Positivity of the {CM} line bundle for families of {K}-stable klt
  {F}ano varieties},
        date={2021},
        ISSN={0020-9910},
     journal={Invent. Math.},
      volume={223},
      number={3},
       pages={811\ndash 894},
         url={https://doi.org/10.1007/s00222-020-00999-y},
      review={\MR{4213768}},
}

\bib{Fuj19}{article}{
      author={Fujita, Kento},
       title={A valuative criterion for uniform {K}-stability of
  {$\Bbb{Q}$}-{F}ano varieties},
        date={2019},
     journal={J. Reine Angew. Math.},
      volume={751},
       pages={309\ndash 338},
}

\bib{GM}{misc}{
      author={Gallardo, Patricio},
      author={Martinez-Garcia, Jesus},
       title={Variations of git quotients package v0.6.13},
         how={\url{https://doi.org/10.15125/BATH-00458.}},
        date={2017},
}

\bib{GMGS21}{article}{
      author={Gallardo, Patricio},
      author={Martinez-Garcia, Jesus},
      author={Spotti, Cristiano},
       title={Applications of the moduli continuity method to log {K}-stable
  pairs},
        date={2021},
        ISSN={0024-6107},
     journal={J. Lond. Math. Soc. (2)},
      volume={103},
      number={2},
       pages={729\ndash 759},
         url={https://doi.org/10.1112/jlms.12390},
      review={\MR{4230917}},
}

\bib{Jiang20}{article}{
      author={Jiang, Chen},
       title={Boundedness of {$\Bbb Q$}-{F}ano varieties with degrees and
  alpha-invariants bounded from below},
        date={2020},
        ISSN={0012-9593},
     journal={Ann. Sci. \'{E}c. Norm. Sup\'{e}r. (4)},
      volume={53},
      number={5},
       pages={1235\ndash 1248},
         url={https://doi.org/10.24033/asens.244},
      review={\MR{4174851}},
}

\bib{KM76}{article}{
      author={Knudsen, Finn~Faye},
      author={Mumford, David},
       title={The projectivity of the moduli space of stable curves. {I}.
  {P}reliminaries on ``det'' and ``{D}iv''},
        date={1976},
        ISSN={0025-5521},
     journal={Math. Scand.},
      volume={39},
      number={1},
       pages={19\ndash 55},
         url={https://doi.org/10.7146/math.scand.a-11642},
      review={\MR{437541}},
}

\bib{KM98}{book}{
      author={Koll\'{a}r, J\'{a}nos},
      author={Mori, Shigefumi},
       title={Birational geometry of algebraic varieties},
      series={Cambridge Tracts in Mathematics},
   publisher={Cambridge University Press, Cambridge},
        date={1998},
      volume={134},
        ISBN={0-521-63277-3},
         url={https://doi.org/10.1017/CBO9780511662560},
        note={With the collaboration of C. H. Clemens and A. Corti, Translated
  from the 1998 Japanese original},
      review={\MR{1658959}},
}

\bib{Kollar04seifert}{misc}{
      author={Koll\'ar, J\'anos},
       title={Seifert ${G}_m$-bundles},
        date={2004},
}

\bib{Kollar13}{book}{
      author={Koll\'{a}r, J\'{a}nos},
       title={Singularities of the minimal model program},
      series={Cambridge Tracts in Mathematics},
   publisher={Cambridge University Press, Cambridge},
        date={2013},
      volume={200},
        ISBN={978-1-107-03534-8},
         url={https://doi.org/10.1017/CBO9781139547895},
        note={With a collaboration of S\'{a}ndor Kov\'{a}cs},
      review={\MR{3057950}},
}

\bib{Kollar19families}{misc}{
      author={Koll\'ar, J\'anos},
       title={Families of divisors},
        date={2019},
}

\bib{Kollar23}{book}{
      author={Koll\'{a}r, J\'{a}nos},
       title={Families of varieties of general type},
      series={Cambridge Tracts in Mathematics},
   publisher={Cambridge University Press, Cambridge},
        date={2023},
      volume={231},
        ISBN={978-1-009-34610-8},
        note={With the collaboration of Klaus Altmann and S\'{a}ndor J.
  Kov\'{a}cs},
      review={\MR{4566297}},
}

\bib{Li17}{article}{
      author={Li, Chi},
       title={K-semistability is equivariant volume minimization},
        date={2017},
        ISSN={0012-7094},
     journal={Duke Math. J.},
      volume={166},
      number={16},
       pages={3147\ndash 3218},
         url={https://doi.org/10.1215/00127094-2017-0026},
      review={\MR{3715806}},
}

\bib{LWX21}{article}{
      author={Li, Chi},
      author={Wang, Xiaowei},
      author={Xu, Chenyang},
       title={Algebraicity of the metric tangent cones and equivariant
  {K}-stability},
        date={2021},
        ISSN={0894-0347},
     journal={J. Amer. Math. Soc.},
      volume={34},
      number={4},
       pages={1175\ndash 1214},
         url={https://doi.org/10.1090/jams/974},
      review={\MR{4301561}},
}

\bib{LX14}{article}{
      author={Li, Chi},
      author={Xu, Chenyang},
       title={Special test configuration and {K}-stability of {F}ano
  varieties},
        date={2014},
        ISSN={0003-486X},
     journal={Ann. of Math. (2)},
      volume={180},
      number={1},
       pages={197\ndash 232},
         url={https://doi.org/10.4007/annals.2014.180.1.4},
      review={\MR{3194814}},
}

\bib{LX19}{article}{
      author={Liu, Yuchen},
      author={Xu, Chenyang},
       title={K-stability of cubic threefolds},
        date={2019},
        ISSN={0012-7094},
     journal={Duke Math. J.},
      volume={168},
      number={11},
       pages={2029\ndash 2073},
         url={https://doi.org/10.1215/00127094-2019-0006},
      review={\MR{3992032}},
}

\bib{LXZ22}{article}{
      author={Liu, Yuchen},
      author={Xu, Chenyang},
      author={Zhuang, Ziquan},
       title={Finite generation for valuations computing stability thresholds
  and applications to {K}-stability},
        date={2022},
        ISSN={0003-486X},
     journal={Ann. of Math. (2)},
      volume={196},
      number={2},
       pages={507\ndash 566},
         url={https://doi.org/10.4007/annals.2022.196.2.2},
      review={\MR{4445441}},
}

\bib{LZ22}{article}{
      author={Liu, Yuchen},
      author={Zhuang, Ziquan},
       title={On the sharpness of {T}ian's criterion for {K}-stability},
        date={2022},
        ISSN={0027-7630},
     journal={Nagoya Math. J.},
      volume={245},
       pages={41\ndash 73},
         url={https://doi.org/10.1017/nmj.2020.28},
      review={\MR{4413362}},
}

\bib{LZ23}{article}{
      author={Loginov, Konstantin},
      author={Zhou, Chuyu},
       title={Boundedness of log {F}ano pairs with certain {K}-stability},
        date={2024},
        ISSN={1073-7928},
     journal={Int. Math. Res. Not. IMRN},
      number={20},
       pages={13281\ndash 13294},
         url={https://doi.org/10.1093/imrn/rnae202},
      review={\MR{4811688}},
}

\bib{MM93}{incollection}{
      author={Mabuchi, Toshiki},
      author={Mukai, Shigeru},
       title={Stability and {E}instein-{K}\"{a}hler metric of a quartic del
  {P}ezzo surface},
        date={1993},
   booktitle={Einstein metrics and {Y}ang-{M}ills connections ({S}anda, 1990)},
      series={Lecture Notes in Pure and Appl. Math.},
      volume={145},
   publisher={Dekker, New York},
       pages={133\ndash 160},
      review={\MR{1215285}},
}

\bib{OSS16}{article}{
      author={Odaka, Yuji},
      author={Spotti, Cristiano},
      author={Sun, Song},
       title={Compact moduli spaces of del {P}ezzo surfaces and
  {K}\"{a}hler-{E}instein metrics},
        date={2016},
        ISSN={0022-040X},
     journal={J. Differential Geom.},
      volume={102},
      number={1},
       pages={127\ndash 172},
         url={http://projecteuclid.org/euclid.jdg/1452002879},
      review={\MR{3447088}},
}

\bib{SS17}{article}{
      author={Spotti, Cristiano},
      author={Sun, Song},
       title={Explicit {G}romov-{H}ausdorff compactifications of moduli spaces
  of {K}\"{a}hler-{E}instein {F}ano manifolds},
        date={2017},
        ISSN={1558-8599},
     journal={Pure Appl. Math. Q.},
      volume={13},
      number={3},
       pages={477\ndash 515},
         url={https://doi.org/10.4310/pamq.2017.v13.n3.a5},
      review={\MR{3882206}},
}

\bib{Tha96}{article}{
      author={Thaddeus, Michael},
       title={Geometric invariant theory and flips},
        date={1996},
        ISSN={0894-0347},
     journal={J. Amer. Math. Soc.},
      volume={9},
      number={3},
       pages={691\ndash 723},
         url={https://doi.org/10.1090/S0894-0347-96-00204-4},
      review={\MR{1333296}},
}

\bib{Xu20}{article}{
      author={Xu, Chenyang},
       title={A minimizing valuation is quasi-monomial},
        date={2020},
        ISSN={0003-486X},
     journal={Ann. of Math. (2)},
      volume={191},
      number={3},
       pages={1003\ndash 1030},
         url={https://doi.org/10.4007/annals.2020.191.3.6},
      review={\MR{4088355}},
}

\bib{XZ20b}{article}{
      author={Xu, Chenyang},
      author={Zhuang, Ziquan},
       title={On positivity of the {CM} line bundle on {K}-moduli spaces},
        date={2020},
        ISSN={0003-486X},
     journal={Ann. of Math. (2)},
      volume={192},
      number={3},
       pages={1005\ndash 1068},
         url={https://doi.org/10.4007/annals.2020.192.3.7},
      review={\MR{4172625}},
}

\bib{Zhou23b}{misc}{
      author={Zhou, Chuyu},
       title={On the shape of {K}-semistable domain and wall crossing for
  {K}-stability},
        date={2023},
}

\bib{Zhou23}{article}{
      author={Zhou, Chuyu},
       title={On wall-crossing for {K}-stability},
        date={2023},
        ISSN={0001-8708},
     journal={Adv. Math.},
      volume={413},
       pages={Paper No. 108857, 26},
         url={https://doi.org/10.1016/j.aim.2022.108857},
      review={\MR{4533746}},
}

\bib{Zhou21a}{article}{
      author={Zhou, Chuyu},
       title={Log {K}-stability of {GIT}-stable divisors on {F}ano varieties},
        date={2024},
        ISSN={1073-2780},
     journal={Math. Res. Lett.},
      volume={31},
      number={4},
       pages={1249\ndash 1262},
         url={https://doi.org/10.4310/mrl.241119010936},
      review={\MR{4837550}},
}

\bib{Zhou23a}{article}{
      author={Zhou, Chuyu},
       title={On {K}-semistable domains---more examples},
        date={2024},
        ISSN={0129-167X},
     journal={Internat. J. Math.},
      volume={35},
      number={2},
       pages={Paper No. 2350103, 30},
         url={https://doi.org/10.1142/s0129167x23501033},
      review={\MR{4712669}},
}

\bib{Zhuang24}{article}{
      author={Zhuang, Ziquan},
       title={On boundedness of singularities and minimal log discrepancies of
  {K}oll\'{a}r components, {II}},
        date={2024},
        ISSN={1465-3060},
     journal={Geom. Topol.},
      volume={28},
      number={8},
       pages={3909\ndash 3934},
         url={https://doi.org/10.2140/gt.2024.28.3909},
      review={\MR{4843751}},
}

\end{biblist}
\end{bibdiv}
\end{document}